\documentclass[smallextended]{svjour3}       

\RequirePackage{fix-cm}
\smartqed  
\usepackage{graphicx}
\usepackage{amssymb,amsmath}
\newcommand{\Singular}{\textbf{\textsc{Singular}}}
\newcommand{\Mathematica}{\textbf{\textsc{Mathematica}}}
\newcommand{\QeHopf}{\textbf{\textsc{QeHopf}}}
\newcommand{\Maple}{\textbf{\textsc{Maple}}}
\newcommand{\Redlog}{\textbf{\textsc{Redlog}}}
\newcommand{\QEPCADB}{\textbf{\textsc{QEPCAD B}}}

\newcommand{\Reduce}{\textbf{\textsc{REDUCE}}}
\newcommand{\Macaulay}{\textbf{\textsc{Macaulay}}}

\newcommand{\deriv}[1]{\frac{\mathrm{d}#1}{\mathrm{d}t}}
\newcommand{\delete}[1]{}

\newcommand{\la}{\langle}
\newcommand{\ra}{\rangle}
\newcommand{\be}{\begin{equation}}
\newcommand{\ee}{\end{equation}}
\newcommand{\logicaland}{\,\land\,}

\newcommand{\N}{\mathbb{N}}
\newcommand{\np}{\mathbb{N}_0}
\newcommand{\Q}{\mathbb{Q}}
\newcommand{\R}{\mathbb{R}}
\newcommand{\vv}{{\bf V}}
\newcommand{\xb}{{\bf x}}
\newcommand{\Z}{\mathbb{Z}}

\newcommand{\fs}{f_1,\dots,f_s}
\newcommand{\kxn}{k[x_1,\dots,x_n]}

\newcommand{\alphab}{{\boldsymbol{\alpha}}}
\newcommand{\betab}{{\boldsymbol{\beta}}}

\journalname{Reaction Kinetics, Mechanisms and Catalysis}
\begin{document}

\title{On three genetic repressilator topologies\thanks{Ma\v sa Dukari\'c and
Valery Romanovski  are  supported by the Slovenian Research Agency (program P1-0306) and  by  a Marie Curie International Research Staff Exchange Scheme Fellowship
within the 7th European Community Framework Programme,
FP7-PEOPLE-2012-IRSES-316338. The work has also been partially supported by the Hungarian-Slovenian cooperation projects  T\'ET\_16-1-2016-0070
and BI-HU-17-18-011. Roman Jerala and Tina Lebar are supported by Slovenian Research Agency project J1-6740 and program P4-0176. Tina Lebar is partially supported by the UNESCO-L'OREAL national fellowship "For Women in Science". J\'anos T\'oth also acknowledges the support by
the National Research, Development and Innovation Office (SNN 125739).}
}

\author{Ma\v sa Dukari\' c \and
Hassan Errami \and Roman Jerala \and
Tina Lebar \and
Valery G. Romanovski \and
J\'anos T\'oth$^*$ \and
Andreas Weber
}

\authorrunning{Dukari\' c,
Errami,
Jerala,
Lebar,
Romanovski,
T\'oth,
Weber}

\titlerunning{
On three genetic repressilator topologies
}

\institute{M. Dukari\' c \at
Center for Applied Mathematics and Theoretical Physics\\
University of Maribor, Mladinska 3, SI-2000, Slovenia\\
\email{masa.dukaric@gmail.com}
\and
H. Errami \at
Institut f\"{u}r Informatik II, Universit\"{a}t Bonn, Bonn, Germany\\
\email{errami@informatik.uni-bonn.de}
\and
R. Jerala and T. Lebar \at
Department for Synthetic Biology and Immunology, National Institute of Chemistry\\
Hajdrihova 19, 1000 Ljubljana, Slovenia\\
\email{Tina.Lebar@ki.si and Roman.Jerala@ki.si}
\and
V. G. Romanovski \at
Faculty of Electrical Engineering and Computer Science,
University of Maribor,\\
Koro\v ska cesta 46,  Maribor, SI-2000 Maribor, Slovenia\\
Center for Applied Mathematics and Theoretical Physics,
University of Maribor,\\ Mladinska 3, SI-2000, Slovenia\\
Faculty of Natural Science and Mathematics, University of Maribor\\
Koro\v ska cesta 160, SI-2000 Maribor, Slovenia\\
\email{Valerij.Romanovskij@um.si}
\and
J. T\'oth corresponding author\at
Department of Mathematical Analysis, Budapest
University of Technology and Economics\\
Budapest, Egry J. u. 1., Hungary, H-1111\\
Chemical Kinetics Laboratory, E\"otv\"os Lor\'and University\\
Budapest, P\'azm\'any P. s\'et\'any 1/A., Hungary, H-1117\\
Tel.: +361 463 2314\quad Fax:  +361 463 3172\\
\email{jtoth@math.bme.hu}
\and
A. Weber \at
Institut f\"{u}r Informatik II, Universit\"{a}t Bonn, Bonn, Germany\\
\email{weber@cs.uni-bonn.de}}

\date{Received: date / Accepted: date}

\maketitle

\begin{abstract}
\delete{Previous models of protein synthesis regulation have shown that a high Hill coefficient (exponent)---which defines the cooperativity of transcription factor binding---is of significant importance for a functional circuit. However, , meaning that their Hill coefficient is as low as 1, still, non-linear responses can be achieved by introduction of positive feedback loops.
At present there are no general efficient methods for determining singular points of polynomial or rational systems of ordinary differential equations of high dimension depending on parameters, thus the study of such complex systems is non-trivial.
We anticipate that our computational approach to  studying
steady states  of  polynomial and rational systems of ordinary differential equations  could be used to analyse  models of  various complex chemical and biological networks with a relatively high number of variables.}
Novel mathematical models of three different repressilator topologies are introduced. As designable transcription factors have been shown to bind to DNA non-cooperatively, we have chosen models containing non-co\-op\-er\-a\-tive elements.
The extended topologies involve three additional transcription regulatory elements---which can be easily implemented by synthetic biology---forming positive feedback loops. This increases the number of variables to six, and extends the complexity of the equations in the model.
To  perform our analysis we had to use combinations of modern symbolic algorithms of computer algebra systems \Mathematica\ and \Singular. The study shows that all the three  models have simple dynamics what can also be called regular behaviour: they have a single asymptotically stable steady state with small amplitude oscillations in the 3D case and no oscillation in one of the 6D cases and damping oscillation in the second 6D case.
Using the program \QeHopf\ we were able to exclude the presence of Hopf bifurcation in the 3D system.
\end{abstract}
\section{Introduction}\label{sec:intro}
To understand complex biological systems such as tissues and cells, extensive knowledge of molecular interactions and mechanisms is necessary. However, an important part of understanding biological complexity is also mathematical modeling, which allows researchers to investigate connections between cellular processes and to develop hypotheses for the design of new experiments.

Jacob and Monod \cite{jacobmonod} were the first to present a  model of the regulation of the synthesis of a structural protein. In this model enzyme levels are regulated at the level of transcription. Specific proteins are produced which repress the transcription of the DNA to its product (mRNA -- messenger ribonucleic acid), which is  translated into $\beta$-galactosidase,
an enzyme for degradation  of  galactose into simple sugars.

Shortly after Jacob and Monod, Goodwin \cite{goodwin} proposed the first mathematical model of a more complex biological system, a genetic oscillator. The simplest formulation of the Goodwin model involves a single gene that represses its own transcription via a negative feedback loop and uses three variables, $ x, y$  and $z$, where
$x$ denotes the quantity of mRNA, $y$ stands for the quantity of the repressor protein, and $z$ is the quantity of the product, which acts as a corepressor and generates the feedback loop by negative control of mRNA production:
\begin{equation}\label{eq:goodwin}
\deriv{x}=\frac{k_1}{k_2+z^n}-k_3x\quad
\deriv{y}=k_4x-k_5y\quad
\deriv{z}=k_6y-k_7z
\end{equation}
All synthesis and degradation rates in the model (represented by coefficients $k_1$  to $k_7$) are linear, with the exception of the repression, which takes the form of a sigmoidal Hill curve.
Here $ n$ denotes the Hill exponent, which may be interpreted in biological systems as the number of ligand molecules that a receptor can bind. At the level of transcriptional regulation, this can be explained by cooperative binding of the repressor protein to DNA (formation of protein-DNA complexes). It has been demonstrated by Griffith \cite{griffith} that limit cycle oscillations can only be obtained when $n>8$, which is unrealistic in terms of transcriptional regulation, where Hill exponents are rarely higher than 3 or 4.

A repressilator is a network of several genes and can be thought of as an extension of the Goodwin oscillator, which is a one-gene repressilator linked by mutual repression in a cyclic topology. Models of cycles of 2--5 genes have first been studied by Fraser and Tiwari \cite{frasertiwari}, while the first experimental implementation of a 3-gene repressilator in a biological system along with a refined model was demonstrated by Elowitz and Leibler \cite{elowitzleibler}. Let  $X_i$ denote the quantity of mRNA and $Y_i$ the quantity of the repressor protein and
let $ \alpha_{0}, \alpha$  and $ \beta$
represent the transcription rate of a repressed promoter, the maximal transcription rate of a free promoter and the ratio of protein and mRNA decay rate, respectively.
Then the  model is given by the equations:
\be \label{eq:leoleib}
\begin{aligned}
\deriv{X_{i}}=& \alpha_0 + \frac \alpha {1 + Y_{i-1}^n} - X_{i}\\
\deriv{Y_{i}}=& -\beta(Y_{i}-X_{i}) \quad(i=1,2,3),
\end{aligned}
\ee
where the indices 0 and 3 are identified. (Let us note that
Elowitz and Leibler write $j$ instead of $i-1$, still they speak about 6 equations.
However, if $i$ and $j$ run independently, then we have $3\times 3+3=12$ equations.
Our modification is also in accordance with our model below.)
In the paper mentioned above Elowitz and Leibler also determine the unique positive stationary point, and  the parameters when the stationary point looses its stability.
They map part of the parameter space,  and  find oscillations \emph{numerically}.
In the Goodwin model, undamped oscillations can only occur when repression is accomplished by the co-repressor $Z$ and never directly by the protein $Y$  \cite{griffith}, probably due to the increased time delay. In the cyclic repressilator by Elowitz and Leibler, oscillations can occur without co-repressors and for Hill exponents  $n$ as low as 2, which is more applicable to biological systems. It also takes into consideration the production of mRNA with a constant rate.

A theoretical solution for the introduction of non-linearity to non-co\-op\-er\-a\-tive biological systems by using transcription factors, where the same proteins are able to repress one gene and activate another gene has been proposed by M\"uller et al. \cite{mullerhofbauerendlerflammwidderschuster} and Widder et al. \cite{widdermaciasole}. Tyler et al. \cite{tylershiuwalton} continue the work by \cite{mullerhofbauerendlerflammwidderschuster} with biologically less restrictive assumptions.
However, such transcription factors are extremely rare in nature and would also be hard to design by directed evolution.
Recently, Lebar et al. \cite{lebarbezeljakgolobjeralakaduncpirsstrazarvuckozupancicbencinaforstnericgaberlonzaricmajerleoblaksmolejerala} have shown that non-linearity can be introduced into a biological system, by introduction of non-cooperative repressors in combination with activators, competing for binding to the same DNA sequence, thus creating a positive feedback loop. In principle, positive feedback loops could be introduced---based on the same DNA binding domain---to build functional repressilator circuits, consisting of non-cooperative repressors.

The above described oscillator circuit was experimentally constructed using three natural repressor proteins, the TetR, LacI and CI repressors. However, construction of functional biological circuits using such natural repressors requires fine-tuning due to their diverse biochemical properties. Furthermore, the low number of well-characterized natural repressor proteins does not enable construction of multiple circuits in a single cell, a fact that may
support the use of stochastic models, cf. e.g. \cite{aranyitoth,erdilente,tothnagypapp}. With the developments in the field of synthetic biology in the recent years, the use of designable repressors has become more and more frequent
\cite{qilarsongilbertdoudnaweissmanarkinlim,%
kianibealebrahimkhanihuhhallxieliweiss,%
lohmuellerarmelsilve,%
garglohmuellersilverarmel,%
congzhoukuocunniffzhang,%
gaberlebarmajerlesterdobnikarbencinajerala,%
lebarbezeljakgolobjeralakaduncpirsstrazarvuckozupancicbencinaforstnericgaberlonzaricmajerleoblaksmolejerala,%
lebarjerala,nissimprtlifridkinperezpineralu}. Such repressors can be designed to bind any DNA sequence due to their modular structure, which can be exploited to eliminate interactions with the cells' genome. Furthermore, they can be designed in almost unlimited numbers and the biochemical properties of individual repressors are very similar, making construction and modeling of synthetic circuits easier. However, the main disadvantage of designable repressors is that they are monomeric, meaning that their binding to DNA is non-cooperative and the Hill exponent $n$  is equal to 1. Under those conditions, the above described models are not expected to produce oscillations. This poses a challenge of introducing non-linearity in complex biological systems, consisting of such repressors.

Equations describing the model of the repressilator by Elowitz and Leibler with only two variables are easy to handle. However, addition of activators to the model increases the number of variables, thus expanding the complexity of the model. Mathematical analysis of systems of equations with a large number of variables is harder,
and can be investigated using deterministic approach based on ordinary differential equations (ODEs) with kinetics which can be either of the mass action type or other, and use the qualitative theory of ordinary differential equations to find bistability, oscillation etc., or calculating solutions numerically. The stochastic description \cite[Chapter 5]{erditoth},\cite[Chapter 10]{tothnagypapp} or \cite{erdilente} usually does not allow to make symbolic calculations because of the complexity of the model. However, in this case one also may turn to the computer to do simulations \cite{sipostotherdi,nagypapptoth}.

In this work, we compare deterministic mathematical models of three different repressilator topologies based on non-cooperative repressors, which can be implemented in biological systems based on designed DNA binding domains such as zinc fingers, TALEs or dCas9/CRISPR fused to activation or repression domains. The models are simplified and consider reactions only at the protein level. The concentration of each repressor and activator over time is described in a separate equation in a system of equations. In the 3D model, we perform the singular point analysis of
 the 3-variable equation system for the basic repressilator topology, consisting of 3 repressors.  In the 6D models we expand the complexity by addition of 3 variables, representing activators.
The study of the system is non-trivial since there are no efficient methods
for determining singular points of polynomial or rational systems
of ODEs of high dimension and depending on parameters. We perform
our analysis using the combinations of  modern  symbolic   algorithms of computer
algebra systems \Mathematica\ \cite{mathematica}  and \Singular\ \cite{deckergreuelpfisterschonemann}, which has not yet been covered in the literature and represents a novel approach in analysis of biological circuits.

Extensive theoretical studies have already been done on the 3D repressilator circuit.
\cite{kuznetsovafraimovich} only treat the special case $\alpha_0=0$ of our model.
In a nonlinear model such a seemingly slight difference may cause qualitative differences.
They also treat saturable degradation, i.e. cases when instead of
$-k_{\mathrm{deg}}x$ one has a term $\frac{-k_{\mathrm{deg}}x}{1+x}.$
They have shown the connection between the evolution of the oscillatory solution and formation of a heteroclinic cycle at infinity.
\cite{dilao} also deals with the $\alpha_0=0$ case, but he derives the usual nonlinear term starting from a mass action model, and using the Michaelis--Menten type approximation.
That author is mainly interested in models with delay.
\cite{guantespoyatos} again assumes $\alpha_0=0,$ and the rational functions are such that both the denominator and the numerator are second degree polynomials. The paper contains no general mathematical statements, only numerical simulations. On the other hand, the mathematically correct paper \cite{mullerhofbauerendlerflammwidderschuster} treats a large class of models including the model by Elowitz and Leibler (but not our models) and give a detailed description of the attractors. Summarizing, none of the models in the literature cover the classes of models we are interested in, and also, the present approach seems to be a novel one from the mathematical point of view and uses models based on recent experiments in synthetic biology.

Note also that \cite{tiggesmarquezlagostellingfussenegger} consider a much more complicated
process, no formulae can be found in the paper itself.
However, its Supplement contains models, delay, stochastic effects, and no qualitative analysis at all. They estimate the parameters of the model.
\cite{thieffrythomas} use the heuristic ideas (\textit{kinetic logic}) of Thomas without a
mathematical treatment.
\section{A 3D model}\label{sec:3D}
First we model the basic repressilator circuit based on non-cooperative repressors, similar to the Elowitz repressilator. The difference compared to the original repressilator model is that here the Hill exponent $n$  is always equal to 1, due to the non-cooperative nature of the repressors. We consider a symmetrical system, where the biochemical properties of all repressors are similar, as expected with designed transcription factors (and \emph{not} to simplify mathematics). We simplify the system to only consider reactions on the protein level. The variables $x,$ $ y$ and $z$ represent the concentrations of each of the repressors, while the parameters
$ \alpha_{0}, \alpha$  and $k_{\mathrm{deg}}$  represent the rate of protein synthesis when the promoter is repressed, the maximal rate of protein synthesis from the free promoter and protein degradation rate, respectively (Figure \ref{fig:RepTop}).
\begin{figure}[!h]\centering
  \includegraphics[width=84mm]{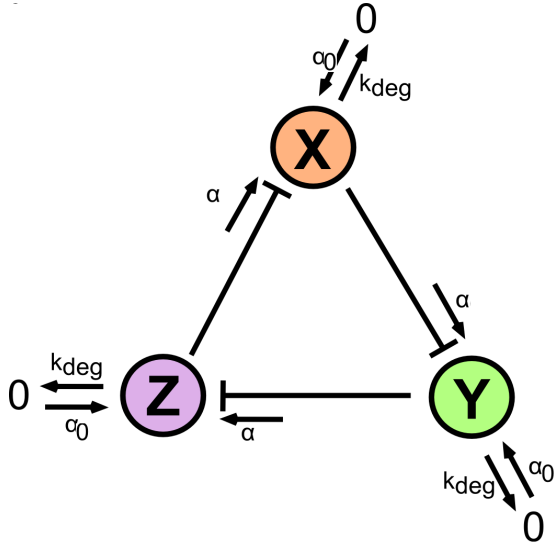}\\
  \caption{The 3D repressilator topology}  \label{fig:RepTop}
\end{figure}
We assume equal rates of synthesis and degradation for all three repressor proteins. Then the concentration of each repressor over time is described by the following equations:
\be \label{ss1}
\begin{aligned}
\frac{d x}{dt}=& \alpha_0 + \frac \alpha  {1 + z} - k_{\mathrm{deg}} x\\
\frac{d y}{dt}=& \alpha_0 + \frac \alpha {1 + x} - k_{\mathrm{deg}} y\\
\frac{d z}{dt}=& \alpha_0 + \frac \alpha {1 + y} - k_{\mathrm{deg}} z,
\end{aligned}
\ee
To simplify the notation we denote
$$
\ s=\alpha_0,\  b=\alpha, \  g=k_{\mathrm{deg}},
$$
where the parameters $s$, $b$ and $g$ are positive real numbers, and the dot denotes the derivative with respect to time.

With  this notation system \eqref{ss1} is written as
\be \label{s1}
\begin{aligned}
\dot x= & s  +\frac b{1 + z} - g x \\
\dot y=& s + \frac b{1 + x} - g y \\
\dot z= & s +\frac {b}{1 + y} - g z.
\end{aligned}
\ee
We are interested  in the  behavior of trajectories of system
\eqref{s1} in the region
$$
D=\{ (x,y,z) :  x>0, \ y>0, \ z>0\}.
$$

System \eqref{s1} has two singular points whose coordinates
contain the expression
$u=\sqrt{4 b g + (g + s)^2}.$ With this we have
\be\label{eq:bandineq}
b= \frac{u^2-(g + s)^2}{4 g}{\quad{\rm  and}\quad u>g+s.}
\ee
Then the steady states  of the system are
$$ A=(x_0,y_0,z_0)=\left(\frac{s- g -  u}{  2 g}, \frac{s- g -  u}{  2 g}, \frac{s- g -  u}{  2 g}\right)
$$
and
$$
B=(x_1,y_1,z_1)=\left( \frac{s+ u-g}{  2 g},\frac{s+ u-g}{  2 g},\frac{s+ u-g}{  2 g}
\right).
$$
The eigenvalues of the Jacobian matrix of system \eqref{s1} at $A$ are
$$
\kappa_1=\frac{2 g u}{g+s-u}, \quad \kappa_{2,3}=-\frac{g (3 g+3 s-u)}{2 (g+s-u)}\pm i \frac{\sqrt{3} g (g+s+u)}{2 (g+s-u)}
$$
and the eigenvalues at $B$ are given by
\be \label{lam}
\lambda_1=
-\frac{2 g u}{g+s+u},\quad \lambda_{2,3}= -\frac{g (3 g+3 s+u)}{2 (g+s+u)}
\pm i \frac{\sqrt{3} g \left( g+s-u\right) }{2 (g+s+u)}.
\ee

We can expect chemically relevant  non-trivial  behavior of trajectories in the domain $D$
if both singular points of the system are located in $D$.
The necessary and sufficient  condition
for this is
\be \label{sem1}
x_0>   0, \  x_1 >  0.
\ee
\delete{
However, $x_0> 0$ alone implies that $s>u+g,$
and this together with $u-g>s$ (see \eqref{eq:bandineq})
shows that these inequalities cannot be fulfilled simultaneously,
thus $A\in D$ can never occur.
Similarly, $x_1<0$ implies $u<g-s,$
and this together with $u>g+s$ (see \eqref{eq:bandineq})
shows that these inequalities cannot be fulfilled simultaneously,
thus $B\in -D$ can never occur.
The only possibility left is that $x_0<0$ and $x_1>0,$
which happens if and only if $s+g<u,$
(which is always the case)
because this inequality automatically implies $s-g<u$ and $g-s<u.$
}
From $u=\sqrt{4bg + (g+s)^2}$, one gets both $u>g$ and $u>s$ (since $b,g$ and $s$ are positive).
As a consequence, $s-g-u = (s-u) - g < 0$ since $s<u.$ Thus, $A$ can be discarded.
Moreover, $s+u-g = s + (u-g) > 0$ since $u>g,$ so $B$ is always in the domain $D$.
Thus, in this case $B$ is in $D$ and $A$ has negative coordinates.
For the eigenvalues \eqref{lam} of the matrix of the linear approximation
of \eqref{s1} at $B$ we have
 $\lambda_1<0$, $\mathrm{Re}\, \lambda_{2,3}<0$, that is,
$B$ is asymptotically stable.

To conclude, in the domain $D=\{ (x,y,z) : x>0, \ y>0, \ z>0\}$ the system can have
only one steady state (point $B$), which is a (locally) asymptotically stable attractor and the trajectories (exponentially) fast approach a neighborhood of the steady state.
In a small neighborhood of it there are damping oscillations, however the amplitude of
oscillations is very small. Why? Because
to obtain oscillations with a large amplitude we need to have at the point $B$ in
$D$ the eigenvalues with $\mathrm{Abs(Re}\, \lambda_{2,3})$ small and $\mathrm{Abs(Im}\, \lambda_{2,3})$
large. However, it can be shown easily that $\mathrm{Abs(Re}\, \lambda_{2,3})<\mathrm{Abs(Im}\, \lambda_{2,3})$ cannot occur.
Thus, this is difficult to achieve in
our system, while it would probably be facilitated in the system with high Hill exponent $n$.
In Fig. \ref{fig:damposc3d} we have chosen the parameters so as to make the difference between
$\mathrm{Abs(Re}\, \lambda_{2,3})$ and $\mathrm{Abs(Im}\, \lambda_{2,3})$ as small as possible.
Fig. \ref{fig:damposc3d} shows the behaviour of the model for a single, specific set of the parameters, but the argument above is symbolic, i.e. valid for all sets of the parameters.
\begin{figure}[!ht]\centering
  \includegraphics[width=84mm]{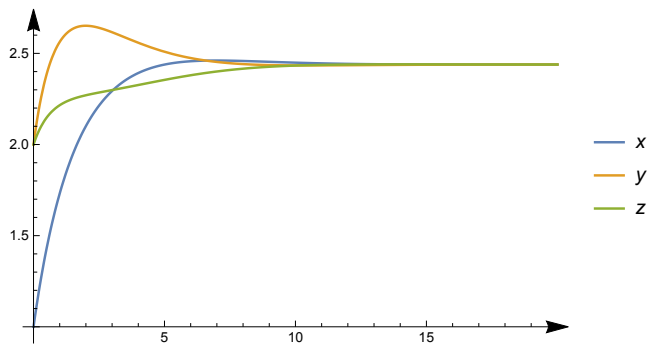}\\
  \caption{ Damping oscillation (overshoot) in the 3D model. $s=0.3, b=4, g=0.6,$ initial concentrations: $(1,2,2).$}
  \label{fig:damposc3d}
\end{figure}

Our calculations above provided an alternative proof of a part of the statement by Allwright \cite{allwright} who has obtained stronger results: he has shown for a class of more general class of models including our one the existence, uniqueness and \emph{global} asymptotic stability of the stationary point.
In order to apply Allwright's results to our model one has to calculate a few quantities, this we will do in the Appendix \ref{subsec:D}.
\section{The forward feedback repressilator 6D model}
By a similar principle that was demonstrated to introduce a non-linear response into a non-cooperative system \cite{lebarbezeljakgolobjeralakaduncpirsstrazarvuckozupancicbencinaforstnericgaberlonzaricmajerleoblaksmolejerala}, we devise a more complex repressilator topology (Figure \ref{fig:doscill2}).
The new system consists of the same repressor topology as the 3D model, but also includes three transcriptional activators, binding to the same DNA targets as the repressors. Each of the activators drives the synthesis of itself and of the next repressor in the cycle. This topology can be implemented in biological systems using a set of three DNA binding domains (X, Y, Z), their combination with an activator (a) or a repressor (r) domain and appropriate binding sites within the three operons.
\begin{figure}[!ht]\centering
\includegraphics[width=84mm]{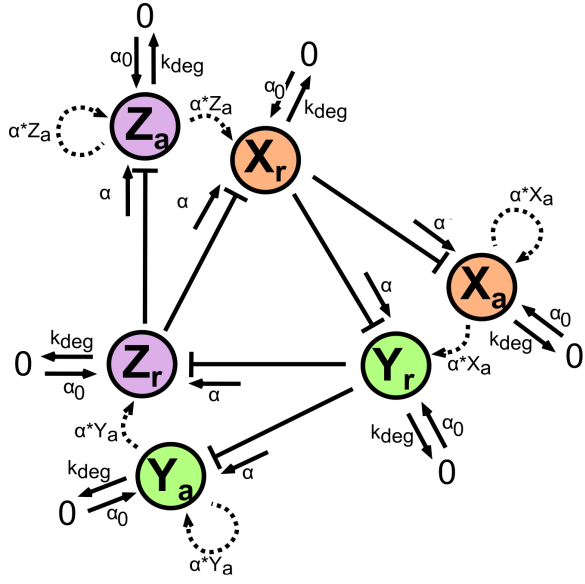}\\
  \caption{A repressilator topology, involving activators, driving the synthesis of the next repressor in the cycle.}
  \label{fig:doscill2}
\end{figure}
The new topology therefore includes 6 variables: the concentration---denoted by the corresponding lowercase letters---of the
 3 repressors ($X_{r}, Y_{r}$ and $Z_{r}$) and 3 activators ($X_{a}, Y_{a}$ and $ Z_{a}$).
The Hill exponent $n$  is always equal to 1, the parameters $\alpha_{0}, \alpha$ and $k_{\mathrm{deg}}$  represent the rate of protein synthesis when the promoter is repressed, the rate of protein synthesis from the free promoter and protein degradation rate, respectively. We assume equal rates of synthesis and degradation for all repressor and activator proteins. In this case, the protein synthesis rate is considered maximal when the activator is bound to the promoter, so concentration of repressors and activators over time is given as:
\be \label{s6o}
\begin{aligned}
 \frac{dx_r}{dt} = & \alpha_0 + \alpha z_a/(1 + z_r + z_a ) -  k_{\mathrm{deg}}   x_r, \\
\frac{dz_a}{dt}  = & \alpha_0 + \alpha z_a/(1 + z_r + z_a ) -  k_{\mathrm{deg}}  z_a,\\
\frac {dy_r }{dt}  = & \alpha_0 + \alpha x_a/(1 + x_r + x_a) -  k_{\mathrm{deg}} y_r \\
\frac{dx_a}{dt}  = & \alpha_0 + \alpha x_a/(1 + x_r + x_a) -  k_{\mathrm{deg}}  x_a, \\
\frac {dz_r }{dt} = & \alpha_0 + \alpha y_a/(1 + y_r + y_a) -  k_{\mathrm{deg}}  z_r,\\
\frac{ dy_a } {dt} = & \alpha_0 + \alpha y_a/(1 + y_r + y_a) -  k_{\mathrm{deg}}  y_a.
\end{aligned}
\ee

Introducing the notation
$$
x_1=x_{r},\ x_2=
z_{a}, x_3= y_{r}, \ x_4= x_{a},\ x_5= z_{r}, \ x_6= y_{a},
\ s= \alpha_{0}, \ b=\alpha, \ g = k_{\mathrm{deg}}
$$
where $b, g, s >0$ we rewrite system \eqref{s6o} in the form
\be \label{s6}
\begin{aligned}
\dot x_1= &s - g x_1 + \frac{b x_2}{1 + x_2 + x_5}=X(x_1,x_2,x_5)
\\
\dot x_2= & s - g x_2 + \frac {b x_2}{1 + x_2 + x_5}=X(x_2,x_2,x_5)
\\
\dot x_3= & s - g x_3 + \frac{b x_4}{1 + x_1 + x_4}=X(x_3,x_4,x_1)
\\
\dot x_4=&s - g x_4 + \frac{b x_4}{1 + x_1 + x_4}=X(x_4,x_4,x_1)
\\
\dot x_5=& s - g x_5 + \frac{b x_6}{1 + x_3 + x_6}=X(x_5,x_6,x_3)
\\
\dot x_6= & s - g x_6 + \frac{b x_6}{1 + x_3 + x_6}=X(x_6,x_6,x_3)
\end{aligned}
\ee
with
\be\label{X}
X(u,v,w):=s-gu+\frac{bv}{1+v+w}.
\ee

From the first two equations of \eqref{s6} we obtain that for any
steady state $(x_1,x_2,\dots, x_6)$ of the system it should be
 that $x_1=x_2$.  Similarly, two other pairs of equations
\eqref{s6} yield
 that $x_3=x_4, x_5=x_6.$ Thus, the  simplified stationary point
 equations  are:
\begin{eqnarray}
0&= &s - g x_1 + \frac{b x_1}{1 + x_1 + x_5}=X(x_1,x_1,x_5)\label{eq:st1}
\\
0&= & s - g x_3 + \frac{b x_3}{1 + x_1 + x_3}=X(x_3,x_3,x_1)\label{eq:st2}
\\
0&=& s - g x_5 + \frac{b x_5}{1 + x_3 + x_5}=X(x_5,x_5,x_3).\label{eq:st3}
\end{eqnarray}
We first  look for  steady states of system \eqref{s6}
using the  routine \texttt{Solve} of \Mathematica\ and
we find 8 steady states.
Two of them are
\be \label{F}
F=(f,f,f,f,f,f), \quad {\rm where \quad } f= \frac{\sqrt{(b-g+2 s)^2+8 g s}+b-g+2 s}{4
   g}
 \ee
and
\be
\label{H}
H=(h,h,h,h,h,h), \quad {\rm where \quad } h  = -\frac{\sqrt{(b-g+2 s)^2+8 g s}-b+g-2 s}{4
   g}.
  \ee
 However, coordinates of the other steady states are given by long cumbersome expressions which are not convenient to analyse. (If one applies \texttt{Simplify} or even \texttt{FullSimplify} the result of \texttt{LeafCount} is more than thirteen thousand.) Thus, we choose another approach to finish.

Chemically relevant steady states should satisfy the conditions
\begin{multline} \label{6semi}
X(x_1,x_1,x_5)= X(x_3,x_3,x_1)=X(x_5,x_5,x_3)=0,\\ s>0,\ g>0,\
 b>0,\ x_1>0,\ x_3>0, \ x_5>0 .
\end{multline}
System \eqref{6semi} is  a so-called semi-algebraic system
(since it contains not only algebraic equations $X(x_1,x_1,x_5)= X(x_3,x_3,x_1)=X(x_5,x_5,x_3)=0$, but also inequalities).
Nowadays powerful algorithms to solve such systems have been developed
and implemented in many computer algebra systems.
In particular, in  \Mathematica\
the routine
  \texttt{ Reduce }  can be applied to finding solutions  of
  semi-algebraic systems.
For algebraic functions
\texttt{Reduce}  constructs equivalent purely polynomial systems
and then  uses cylindrical algebraic decomposition (CAD) introduced by Collins in \cite{collins}    for real domains and Gr\"{o}bner basis methods for complex domains.

To simplify computations  we first clear
the denominators on the right hand side of \eqref{eq:st1}--\eqref{eq:st3} obtaining the polynomials
\begin{eqnarray*}
  f_1&:=&s + b x_1 - g x_1 + s x_1 - g x_1^2 + s x_5 - g x_1 x_5 \\
  f_1&:=&s + s x_1 + b x_3 - g x_3 + s x_3 - g x_1 x_3 - g x_3^2 \\
  f_1&:=&s + s x_3 + b x_5 - g x_5 + s x_5 - g x_3 x_5 - g x_5^2 0
\end{eqnarray*}
Solving with \texttt{Reduce} of \Mathematica\
  the semi-algebraic system
\begin{equation}
f_1 =  f_3 =  f_5 = 0,  x_1 > 0,
    x_3 > 0,  x_5 > 0,  s > 0,  g > 0,  b > 0,\label{semiin}
\end{equation}
with respect to $
 x_1, x_3, x_5, s, b,
   g$
we obtain the solution
\begin{equation}
 x_1 > 0,  x_1 =x_3 = x_5,   s > 0, b > 0,
 g =\frac s{x_1} +\frac { b}{1 + x_1 + x_5}.\label{semiout}
\end{equation}
The input command and the output are given in Appendix \ref{subsec:B}. The exact result may slightly differ depending on the version you use, but nevertheless, it always implies the essential relation that $x_1 =x_3 = x_5.$

Solving the last equation for $x_1$ we obtain two solutions:
$$
\frac{\sqrt{(b-g+2 s)^2+8 g s}+b-g+2 s}{4
   g}
\mathrm{ and }
\frac{-\sqrt{(b-g+2 s)^2+8 g s}+b-g+2 s}{4
   g}.
$$
However in the second case $x_1$ is negative, so the only
steady state whose coordinates satisfy \eqref{6semi} is the point
$F$ defined by \eqref{F}.

Computing the eigenvalues of the Jacobian matrix of system
\eqref{s6} at  $F$ we find that they are
\begin{eqnarray*}
\kappa_{1,2,3}&=&-g, \quad \kappa_4 = -g+\frac{b}{(1 + 2 f)^2} \\
\kappa_{5,6}&=&-g+\frac{b(2 + 3f)}{2 (1 + 2 f)^2}\pm i \frac{\sqrt{3} b f}{2 (1 + 2 f)^2},
\end{eqnarray*}
where $f$ is defined by \eqref{F}.
A short calculation shows that all eigenvalues of the Jacobian matrix have negative real parts yielding
that  $F$ is asymptotically  stable.
Thus, we have proven the following result.
\begin{theorem}
Point $F$ is the only positive stationary point of system \eqref{s6}
and it is asymptotically stable.
\end{theorem}
\section{The backward feedback repressilator 6D model}
Due to the absence of oscillations in the above described model we next consider a repressilator topology with activators wired to activate transcription of the previous repressor in the cycle (Figure 4). The notations of the variables and the constants are the same as in the previous 6D model. Therefore, the concentrations of repressors and activators over time are as follows:
\be \label{s6sec}
\begin{aligned}
\dot x_1= & s - g x_1 + b x_4/(1 + x_4 + x_5)
=X(x_1,x_4,x_5)
\\
\dot x_2= & s - g x_2 + b x_4/(1 + x_4 + x_5)
=X(x_2,x_4,x_5)
\\
\dot x_3= & s - g x_3 + b x_6/(1 + x_1 + x_6)
=X(x_3,x_6,x_1)
\\
\dot x_4=&s - g x_4 + b x_6/(1 + x_1 + x_6)
=X(x_4,x_6,x_1)
\\
\dot x_5=& s  - g x_5 + b x_2/(1 + x_2 + x_3)
=X(x_5,x_2,x_3)
\\
\dot x_6= & s  - g x_6 + b x_2/(1 + x_2 + x_3)
=X(x_6,x_2,x_3)
\end{aligned}
\ee
with $X(u,v,w):=s-gu+\frac{bv}{1+v+w}$,
that is, $X$ is again defined by \eqref{X}, but the right-hand-sides do depend on three variables, differently form the previous case.
\begin{figure}\centering
  \includegraphics[width=154mm]{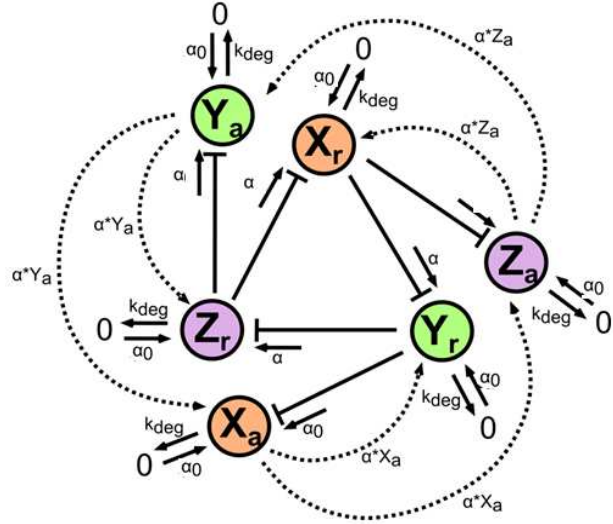}\\
  \caption{The backward feedback repressilator 6D model: repressilator topology including transcriptional activators, driving the synthesis of the previous repressor in the cycle.}
  \label{fig:doscill1}
\end{figure}
\subsection{Steady states of the model}
From \eqref{s6sec} it is easily seen that any stationary point
of \eqref{s6sec} should fulfil
 $x_2=x_1, x_4=x_3, x_6=x_5$. Then,  similarly as in the case of system \eqref{s6}, computing
with \Mathematica\ we find that the system has singular points $F$ and $H$
defined by \eqref{F} and \eqref{H} and  a few other singular
points whose coordinates are given by cumbersome expressions,
which are not suitable for further analysis.
Therefore, again we proceed using the previous ideas.

The chemically relevant
 steady states of system \eqref{s6sec} are solutions to the semi-algebraic system
\be \label{6sec_semi}
f_1=f_2 =f_3=0,\ s>0,\ g>0,\ b>0,\ x_1>0, x_3>0, x_5>0
\ee
where
$$
\begin{aligned}
f_1= &  s - g x_1 + b x_3 + s x_3 - g x_1 x_3 + s x_5 -
 g x_1 x_5 ,\\
f_2= &   s + s x_1 - g x_3 - g x_1 x_3 + b x_5 + s x_5 -
 g x_3 x_5  ,  \\
f_3= &  s + b x_1 + s x_1 + s x_3 - g x_5 - g x_1 x_5 -
 g x_3 x_5
\end{aligned}
$$
(that is, $f_1= X(x_1,x_3,x_5   )(1 + x_3 + x_5),
f_2 = X(x_3,x_5,x_1) (1 + x_1 + x_5), \ f_3  =X(x_5,x_1,x_3) (1 + x_1 + x_3)
$).
But unlike the case of the previous model, we were
able to solve system \eqref{6sec_semi} neither  with \texttt{Reduce} nor \texttt{Solve}
of \Mathematica. (\texttt{Solve} provides five roots, most of them in uselessly complicated form.)
It appears that the reason is
that in the previous model   the steady states were determined
 from   the  system
$$
X(x_1,x_1,x_5)= X(x_3,x_3,x_1)=X(x_5,x_5,x_3)=0,
$$
where each equation depended only on two variables,
whereas in the present case they are to be determined from the system
$$
X(x_1,x_3,x_5)=X(x_3,x_5,x_1)=X(x_5,x_1,x_3)=0,
$$
where each equation depends on three variables,
so the latter system is more complicated.


To find the steady states of system \eqref{s6sec} we
use the  computer algebra system \Singular\ \cite{deckerlaplagnepfisterschonemann,deckergreuelpfisterschonemann}.
We  look for solutions of system
\be\label{fss}
f_1(x_1,x_3,x_5,s,g,b)=  f_2(x_1,x_3,x_5,s,g,b)=f_3(x_1,x_3,x_5,s,g,b)=0.
\ee

The polynomials $f_1,f_2,f_3$ are polynomials of six variables
with rational coefficients, that is, they are polynomials
of the ring $\Q[ s,b,g,x_1,x_3,x_5]$. In \Singular\ the ring of such
polynomials can be  declared as

 \centerline{
 \texttt{ring r=0,(s,b,g,x1,x3,x5),(lp)}),}
\noindent  where \texttt{r} is the name of the ring,  $0$ is the characteristic
of the field of rational numbers $\Q$,  and  \texttt{lp}
 means that Gr\"{o}bner basis calculations should
 be performed using the lexicographic ordering.

Let $I$ be the ideal generated by $f_1,f_2,f_3$
in  $\Q[ s,b,g,x_1,x_3,x_5]$, that is,
\be \label{I}
I=\la f_1,f_2, f_3 \ra.
\ee
The set of solutions of system \eqref{fss} is the variety $V(I)$
of $I$ (the zero set of all polynomials from $I$).
(We give definitions and some facts about polynomial ideals and their varieties in Appendix \ref{subsec:A}.)
 Then, applying the routine \texttt{minAssGTZ} of
 \cite{deckergreuelpfisterschonemann}, which computes
 minimal associate primes of polynomial ideals using
 the algorithm of \cite{giannitragerzacharias},  we find that the variety of
 $I$ consists of  three  components,
 \begin{equation}
\vv(I)=\vv(I_1)\cup \vv(I_2)\cup \vv(I_3),\label{star}
\end{equation}
where $I_1,I_2,I_3$ are the ideals written under
[1]:, [2]: and [3]:, respectively, in Appendix \ref{subsec:C}.

Since $I_1=\la
   x_3-x_5,
   x_1-x_5,
   2 s x_5+s+b x_5-2 g x_5^2-g x_5\ra
   $
it is easily seen that the variety $\vv(I_1)$ consists
of two points $F$ and $H$ defined by \eqref{F} and \eqref{H}, respectively.
   From the equations for the  third component  we have
   $s=g=b=0$, so the system degenerates.

However, the polynomials defining  the second  component are complicated  and  difficult to  analyse, so we are  unable to extract useful description of the component
from these polynomials.

Fortunately, there is a slightly different way to treat
  the problem of solving  system \eqref{fss}. Namely, we can treat
polynomials
$$  f_1(x_1,x_3,x_5,s,g,b), \
  f_2(x_1,x_3,x_5,s,g,b), \ f_3(x_1,x_3,x_5,s,g,b)
$$
as polynomials of $ x_1, x_2, x_3$ depending on parameters
$s,g,b$ (which is in agreement with the meaning of $s,g,b$ in  differential  system \eqref{s6sec}).

To do so,  we declare the ring as

\centerline{\texttt{ring r=(0,s,b,g),(x1,x3,x5),(lp)},}
\noindent  where \texttt{r} is the name of the ring,  \texttt{(0,s,b,g)}
 means that the computations should be performed in the field
 of characteristic \texttt{0} and \texttt{s,b,g} should be treated as parameters, and,
 as above, \texttt{lp} means that Gr\"{o}bner basis calculations should
 be performed using the lexicographic ordering.

 Computing with \texttt{minAssGTZ} the minimal associate primes of the ideal
 $
J=\la f_1,f_2, f_3 \ra
$ (which looks as $I$ but now   it is  considered as the ideal of the ring
$
  \Q(s,b,g)[x_1,x_3,x_5]
$)
 we obtain that they are
$$
J_1=\la h_1,  h_2,  h_3\rangle
$$
with
\be \label{h123}
\begin{aligned}
h_1=& (2 s g^3+b g^3+g^4) x_5^3+(2 s^2 g^2+2 s b g^2+5 s g^3+2 b^2 g^2+2 b g^3+2 g^4) x_5^2\\
&+(-2 s^3 g-3 s^2 b g-s^2 g^2-3 s b^2 g-2 s b g^2+2 s g^3-b^3 g+g^4) x_5\\
&+(-2 s^4-4 s^3 b-5 s^3 g-5 s^2 b^2-8 s^2 b g-4 s^2 g^2-3 s b^3-7 s b^2 g\\
&-5 s b g^2-s g^3-b^4-2 b^3 g-2 b^2 g^2-b g^3),\\
h_2= &(2 s b g+b^2 g+b g^2) x_3+(-2 s g^2-b g^2-g^3) x_5^2\\
&+(s b g-2 s g^2-b^2 g-g^3) x_5\\
&+(2 s^3+4 s^2 b+3 s^2 g+4 s b^2+6 s b g+s g^2+2 b^3+2 b^2 g+2 b g^2)\\
h_3= &(2 s b g+b^2 g+b g^2) x_1+(2 s g^2+b g^2+g^3) x_5^2\\
&+(s b g+2 s g^2+2 b^2 g+b g^2+g^3) x_5\\
&+(-2 s^3-2 s^2 b-3 s^2 g-2 s b^2-s b g-s g^2)
\end{aligned}
 \ee
 and
$$ J_2=\la
 2 g x_5^2+(g-2 s-b) x_5-s,
   x_1-x_5,
   x_3-x_5\ra.
   $$

 So the variety of the ideal consists of two components
$$
{\bf V}(J)={\bf V}(J_1)\cup {\bf V}(J_2).
$$

   Clearly, the variety ${\bf V}(J_2) $ considered as a  variety in
   $\mathbb{R}^3$  consists of two points
 $F$ and $H$ defined by \eqref{F} and \eqref{H}.

Chemically relevant steady states in the component ${\bf V}(J_1)$  are determined from the
semi-algebraic system
 \be \label{sas6}
 b>0, \ g>0, \ s>0,x_1>0, x_3>0 \ x_5>0,  h_1=0, h_2=0, h_3=0.
 \ee

 Solving  system \eqref{sas6} with  \texttt{Reduce}  we find  that it  has no solution
(the command \texttt{Reduce} returns \texttt{False} as the output).

Using the  analysis performed above  we can prove  the following result.
\begin{theorem}
The only steady state of system \eqref{s6} with positive coordinates  is the  point $F$ defined by \eqref{F}.
\end{theorem}
\begin{proof}
As we have shown above the only point from the variety
$\vv (J)$ satisfying the condition
 \be \label{csas6}
 b>0, \ g>0, \ s>0,x_1>0, x_3>0 \ x_5>0
 \ee
 is the point $F$ defined by \eqref{F}.

 However, the  complete set of  steady states of system
 \eqref{s6sec} is determined from the variety $\vv(I)$
 of the ideal $I$  defined by \eqref{I}.
Thus to prove the theorem it is sufficient to show that
  $\vv (I)$ is a subset of  $\vv(J)$.
The first components of $\vv (I)$ and the second component of $\vv(J)$ are the same,
the third component of $\vv (I)$ is the variety $\vv(I_3)$ of the ideal
 $I_3=\la s, b , g\ra$. Obviously,
if
 $s=b=g=0$ then all polynomials $h_1, h_2, h_3$ vanish,
 that means, $\vv (I_3)$ is subset of $\vv (J)$.
    So, we have to compare  the second
 components of the decompositions of  $\vv (I)$ and $\vv(J)$,
 that is, $\vv(H)$ and $\vv(G)$,
 where $H=\la h_1, h_2, h_3 \ra $ with
 $h_1, h_2, h_3$   defined by \eqref{h123} and
 $G=\la g_1, \dots, g_{11}\ra$
 where by $g_1, \dots, g_{11}$ we denote polynomials
 of the second minimal associate prime given in Appendix \ref{subsec:C}.

 First, with the command \texttt{std} of \Singular\
 we compute  Gr\"{o}bner bases of $H$ and $G$, denoting them
 $H_s$ and $G_s$, respectively.
 Then with \texttt{reduce} of \Singular\ we check
 that    $H\subset G$  (since \texttt{reduce($H_s$,$G_s$)}
 returns $0$)  yielding  $\vv(H)\subset \vv(G)$.
$\square$
\end{proof}
 \begin{remark}
Applying the command  \texttt{reduce($G_s,H_s$)} we obtain
  that    $H\subsetneq G$    yielding  $\vv(H)\subset \vv(G)$,
  and  $\vv(H)$ is a strict subset of  $\vv(G)$ (as varieties in $\mathbb{C}^6$).

  We also can find  the precise difference of $\vv(H)$ and $\vv(G)$,
  the set  $\vv(H)\setminus \vv(G)$.
  To this end, we use the fact that
     $$\vv(H)\setminus \vv(G)= \vv(H:G), $$
 where $H:G$ is the quotient of ideals $H$ and $G$ (see e.g. \cite{coxlittleshea} or \cite{romanovskishafer}). In \Singular\
 we compute
 the ideal $H:G$ with the command \texttt{quotient(H,G)}
and then with \texttt{minAssGTZ} we compute the minimal associate primes
of $H:G$ finding that the variety of $H:G$ consists of 5 components:

1) $g=
   s^2+s b+b^2= 0 $

2)
   $b
   =2 s+g=0$

3)
   $3 b-g
 =3 s+2 g=0$

4) $
   b
  =g x_5-s=0$,

5)
   $b
   =g x_5+s+g=0$

Thus,  we see that the varieties $\vv(H)$ and $\vv(G)$ differ only for
the set of parameters which are not relevant for our study: $g=0$ in case 1),
 $b=0$ in cases 2), 4), 5) and in case 3) $s=-2/3 g$ which is impossible
 since $s$ and $g$ are positive.
\end{remark}
\subsection{Stability of the positive steady state}
To study the stability properties of system \eqref{s6sec} near the point $F$ we compute  the characteristic polynomial $p$ of the Jacobian matrix of system
\eqref{s6sec}  at  $F$ and we find that it is given as
\begin{eqnarray*}
  p(u) &=&
\frac{(g + u)^3 }{(1 + 2 f)^6}
   \left(-b+g(1+2f)^2 + u (1+2f)^2\right) \\
   && \left(u^2 (1 + 2 f)^4 + u (1 + 2 f)^2 (b + 2 g (1 + 2 f)^2)\right. \\
   &+& \left.g^2 (1 + 2 f)^4 + b g (1 + 2 f)^2 + b^2 (1 + 3 f + 3 f^2)\right).
\end{eqnarray*}
where $f$ is defined by \eqref{F}. In order to prove that all the roots of the characteristic polynomial have a negative real part it is enough to show that
$-b+g(1+2f)^2>0,$ which can be easily proven, e.g.\ using \texttt{Reduce}.

To sum up, for any $s,b,g >0$ all roots of $p$ have negative real parts.
Therefore, we have proven the following statement.
\begin{theorem}
The only positive steady state $F$ of system \eqref{s6sec} is asymptotically stable.
\end{theorem}
We can get a more precise conclusion about the eigenvalues
of $F$. Computing the discriminant of the second degree factor of the above polynomial we find that  it is
$-3 b^2 (1 + 2 f)^6<0,$
which  means that the polynomial $p$ always has a pair of complex conjugate eigenvalues.

Thus, the matrix of the linear approximation of
\eqref{s6sec} at $F$ always has four negative real eigenvalues
and a pair of complex conjugate eigenvalues with negative real parts.
Consequently Hopf bifurcation is not possible in the system.
We can expect to observe strong damping oscillations near the steady
states if the absolute value of the real parts of the complex
eigenvalues are much less than their imaginary parts.
However our numerical experiments  show that the situation appears to be
just the opposite:  the real parts of the complex
eigenvalues are much larger  than their imaginary parts.
So we can observe only oscillations  which quickly goes to the steady state
(see Fig. \ref{fig:7}).

\begin{figure}[!ht]\centering
  \includegraphics[width=84mm]{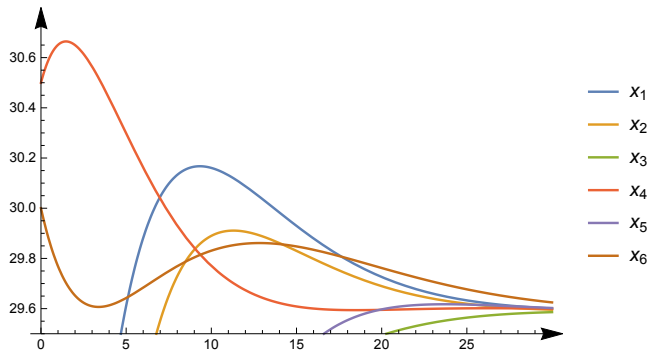}\\
  \caption{Damping oscillation in the 6D model. $s=1, b=10, g=0.2,$ initial concentrations: $(25, 23, 25, 30.5, 21, 30)$}
  \label{fig:7}
\end{figure}

\section{Excluding Hopf Bifurcations by Fully Algorithmic Methods}

We also looked for Hopf bifurcations in the 3D and 6D models using the software
package \QeHopf\ which uses the method of the semi-algebraic
characterization of Hopf bifurcation described  in \cite{elkahouiweber}
(the package is available by request to the authors).
To detect Hopf bifurcation in the models we first generate from the symbolic description of the respective ordinary differential equation a first-order formula
in the language of ordered fields, where our domain is the real numbers.
Specifically,
for a parametrized vector field $f(u,x)$ and the autonomous
ordinary differential system associated with it this semi-algebraic description
 can be expressed by the following first-order
formula:
\begin{eqnarray}
\lefteqn{\exists x (f_{1}(u,x)=0  \logicaland f_{2}(u,x)=0
\logicaland \cdots \logicaland f_{n}(u,x)=0}
\nonumber \\
& &  \logicaland a_{n}>0 \logicaland \Delta_{n-1}(u,x)=0 \logicaland
\Delta_{n-2}(u,x)>0 \logicaland \cdots \logicaland \Delta_{1}(u,x)>0).
\end{eqnarray}
In this formula $a_n$ is $(-1)^{n}$ times the  determinant
of the Jacobian matrix $Df(u,x)$, and $\Delta_{i}(u,x)$ is the $i^{\rm
th}$ Hurwitz determinant of the characteristic polynomial of the
same matrix $Df(u,x)$.
Constraints on parameters are added, and for the rational systems we are considering
one is using the common numerators (adding the condition of non-vanishing denominators).
\QeHopf\ is implemented in \Maple, and the input for the 3D model is as follows:

\begin{footnotesize}
\begin{verbatim}
PP:=diff(x(t),t)= s-g*x(t)+b/(1+z(t)) ;
QQ:=diff(y(t),t)= s-g*y(t) +b/(1+x(t));
RR:=diff(z(t),t)= s-g*z(t)+b/(1+y(t));
fcns:={x(t), y(t) ,z(t)};
params:=[s, g, b];
paramcondlist:=[s>0, g>0, b>0];
funccondlist:=[x(t)>0, y(t)>0, z(t)>0];

DEHopfexistence({PP,QQ,RR}, fcns, params, funccondlist, paramcondlist);
\end{verbatim}
\end{footnotesize}

For the 3D model the generated first-order formula is as follows
\begin{footnotesize}
\begin{verbatim}
informula :=
 ex (vv3,  ex (vv2,  ex (vv1,  ( ( ( 0 < vv1 and  0 < vv2 ) and  0 < vv3 ) and
( ( ( ( ( ( ( s > 0 and b > 0 and g > 0 and
-g*vv1*vv3-g*vv1+s*vv3+b+s = 0 ) and
1+vv3 <> 0 ) and
-g*vv1*vv2-g*vv2+s*vv1+b+s = 0 ) and
1+vv1 <> 0 ) and
-g*vv2*vv3-g*vv3+s*vv2+b+s = 0 ) and
1+vv2 <> 0 ) and
( ( ( 0 < g^3*vv1^2*vv2^2*vv3^2+2*g^3*vv1^2*vv2^2*vv3+2*g^3*vv1^2*vv2*vv3^2
+2*g^3*vv1*vv2^2*vv3^2+g^3*vv1^2*vv2^2
+4*g^3*vv1^2*vv2*vv3+g^3*vv1^2*vv3^2+4*g^3*vv1*vv2^2*vv3
+4*g^3*vv1*vv2*vv3^2+g^3*vv2^2*vv3^2
+2*g^3*vv1^2*vv2+2*g^3*vv1^2*vv3+2*g^3*vv1*vv2^2+8*g^3*vv1*vv2*vv3
+2*g^3*vv1*vv3^2
+2*g^3*vv2^2*vv3+2*g^3*vv2*vv3^2+g^3*vv1^2+4*g^3*vv1*vv2+4*g^3*vv1*vv3
+g^3*vv2^2 +4*g^3*vv2*vv3+g^3*vv3^2+2*g^3*vv1
+2*g^3*vv2+2*g^3*vv3+b^3+g^3 and
0 < (1+vv2)^2*(1+vv3)^2*(1+vv1)^2 ) and
8*g^3*vv1^2*vv2^2*vv3^2+16*g^3*vv1^2*vv2^2*vv3+16*g^3*vv1^2*vv2*vv3^2
+16*g^3*vv1*vv2^2*vv3^2+8*g^3*vv1^2*vv2^2+32*g^3*vv1^2*vv2*vv3
+8*g^3*vv1^2*vv3^2+32*g^3*vv1*vv2^2*vv3+32*g^3*vv1*vv2*vv3^2
+8*g^3*vv2^2*vv3^2+16*g^3*vv1^2*vv2+16*g^3*vv1^2*vv3+16*g^3*vv1*vv2^2
+64*g^3*vv1*vv2*vv3+16*g^3*vv1*vv3^2+16*g^3*vv2^2*vv3+16*g^3*vv2*vv3^2
+8*g^3*vv1^2+32*g^3*vv1*vv2+32*g^3*vv1*vv3+8*g^3*vv2^2
+32*g^3*vv2*vv3+8*g^3*vv3^2
+16*g^3*vv1+16*g^3*vv2+16*g^3*vv3-b^3+8*g^3 = 0 ) and
(1+vv2)^2*(1+vv3)^2*(1+vv1)^2 <> 0 ) ) ) )  )  ) ;
\end{verbatim}
\end{footnotesize}

The system variables became quantified variables and have been renamed to \verb!vv1!, \verb!vv2!, and \verb!vv3!,
and the existential quantification is expressed using the syntax of the package \Redlog\ \cite{dolzmannsturm,sturmredlog},
which had been originally driven by the efficient implementation of quantifier elimination based on virtual substitution methods.
Applying quantifier elimination to the formula yields in principle a quantifier-free semi-algebraic description
of the parameters for which Hopf bifurcation fixed points exist.

If one suspects that there is no Hopf bifurcation fixed point or one just wants to assert that there is one,
then one can apply quantifier elimination to the existential closure of our generated formula.
If all variables and parameters are known to be positive, the technique of positive quantifier elimination
can be used \cite{sturmweberabdelrahmanelkahoui}.
\QeHopf\ uses for the quantifier elimination \Redlog, which  can use \QEPCADB\ \cite{brown}
for formula simplification and as \textit{fallback method}.
However, for the 3D model already \Redlog\ reduces this formula to the equivalent formula \verb!false!,
i.e. for no  parameters
(obeying the positivity condition) a Hopf bifurcation fixed point exists (for positive values).
The needed computation time was less
than 20\,ms.

For the 6D model the fully algorithmic method was not successful, as already the generation
of the formula using Maple failed.
\section{Discussion}
Synthetic biology is one of the most rapidly developing fields of biology. Synthetic genetic circuits are of high interest due to their possible applications in biosensing, bioremediation, diagnostics, therapeutics, etc.
Genetic oscillators are some of the most studied circuits due to their complexity and the possibility of many different topologies.
Building synthetic genetic oscillators with controllable periods and amplitudes would be of great interest to the synthetic biology field as they could for example potentially be used for treatment of diseases related to the circadian cycle.

The experimental validation of complex systems, such as oscillators, can be technically demanding and time consuming. To this day, there has been only few experimental implementations of synthetic oscillators (\cite{elowitzleibler,tiggesmarquezlagostellingfussenegger}). Hence, mathematical modeling of such systems is highly desirable to reduce the experimental workload. Here, we focus on mathematical modeling of 3-cycle genetic repressilators, which have been extensively studied before. However, our study is focused on models based on non-cooperative transcriptional repressors, meaning that all Hill coefficients are always equal to 1. Different studies have already demonstrated that cooperative binding is necessary to obtain oscillations in repressilator systems (\cite{elowitzleibler,bratsunvolfsontsimringhasty,%
mullerhofbauerendlerflammwidderschuster,%
wangjingchen}). Our 3D model confirms that oscillations in such a system are indeed absent. However, a theoretical study by \cite{tsaichoimapomereningtangferrell} has shown that the range of parameters in which the system produces oscillations can be expanded by including positive interactions, facilitated by transcriptional activators. We additionally model two repressilator topologies, involving 3 transcriptional activators, driving transcription of either the next or the previous repressor in the cycle. (Let us mention that Allwright's results cannot be applied for our 6D models.)

What do offer the general results of formal reaction kinetics for the treatment our models? The differential equations of each of the models can be considered as induced kinetic differential equations of a reversible reaction, therefore existence of the positive stationary state follows from general results \cite{boroswrexistence, tothnagypapp}, see the details in \ref{subsec:frk}.

To summarize our mathematical results, we have shown that for all positive values of parameters $b, g, s$ system
\eqref{s1} has a single positive stationary point which is a globally asymptotically stable attractor.
Furthermore, \eqref{s6} and \eqref{s6sec} have a single stationary state (point $F$ defined by \eqref{F}) in
the domain $x_i>0 , \ (i=1,\dots, 6)$,
which is a locally asymptotically stable attractor.

Comparing the 3D and 6D models we see that the properties
of solutions in the domains, where all phase variables are
positive, are similar. For all the three systems in these domains there
is a unique singular point which is a strong attractor.
In the 3D system, a small overshoot is possible near the steady state,
whereas no oscillations appear in the first 6D model near the steady state.

In both 6D  models the steady state is an attractor: 
in both cases all  eigenvalues of the steady state have negative real parts,
however two eigenvalues are always
complex conjugate, so it is possible
to observe damping oscillations near the steady state, see Fig. \ref{fig:7}.
Thus, the 6D models demonstrate richer dynamics than the 3D models, including the possibility of damped oscillations.

We can also note that these  models, as many others  arising
in the studying of biochemical phenomena, exhibit rather simple dynamics.
It was somewhat surprising because the models are given  by systems
of  differential equations depending on few parameters, and there are  systems
which look  simpler, but exhibit rather complicate, even chaotic,
dynamics. It can be a challenging problem to understand the  reasons
for such simple dynamics. One source of argument may originate in the fact
the models' stationary states are so closely related to stationary states of one linkage class reversible reactions as described in \ref{subsec:frk}.

From the biochemical point of view, the probable reason for the absence of oscillations in the first 6D model is the strength of the activator feedback, which forms a negative feedback loop despite the positive interaction. Nevertheless, different combinations of activators and repressors could result in topologies that produce regular oscillations. Due to the stochasticity of biological systems, stochastic modeling and algorithms could be used to further analyze these topologies.

As to the computational methods: they are based on recent mathematical and algorithmic developments, and can be applied to many different similar problems frequently arising in biochemical studies. Note that theory makes it possible to turn to simpler polynomials than those at the beginning, and also that it is not the same to have a six variable polynomial and to have a three variable polynomial with three parameters.
\section{Appendix}\label{sec:appendix}
\subsection{On the nonlinear term}\label{subsec:nonlinear}
The term $\frac{k_1}{k_2+z^n}$ in \eqref{eq:goodwin} is (from the point of calculations) similar to the one obtained when the Michaelis--Menten kinetics is approximated by Tikhonov method, or to the Holling type kinetics which is often used in population dynamics \cite{kisstoth}. Therefore the methods used above may have applications in reaction kinetics and population biology, as well. The main difference between this term and the reaction rates usually used is that although this rate is always positive, it is not zero if $z$ or $x$ is zero, a general requirement quite often assumed, \cite[p. 613]{volperthudjaev}.
\subsection{Solving systems of polynomial equations}\label{subsec:A}
We give a short summary on the topics of solving polynomial systems.
The interested reader may consult \cite{coxlittleshea,romanovskishafer} for more details.

Let $k[x_1, \ldots, x_n]$ denote the ring of polynomials in $n$ indeterminates with coefficients in the field $k$,
which is  typically the set $\R$ of real numbers or $\mathbb{C}$ of complex numbers.

The problem of finding solutions to  a  system of polynomials
\begin{eqnarray}\label{e:poly.sys}
f_1(x_1,\dots,x_n)&=&0,\nonumber\\
\qquad\quad\ \vdots&&\\
f_m(x_1,\dots,x_n)&=&0\nonumber
\end{eqnarray}
is a challenging mathematical problem.
Such systems  often  have infinitely many solutions, and it is simply impossible to find  them all numerically. Even if system (\ref{e:poly.sys})  has a finite number of solutions, it is still very difficult and often impossible to find  all of them
 numerically without applying methods of computational algebra.

In fact, no regular methods for  solving system (\ref{e:poly.sys})  were known until the mid-sixties of the last century
 when Bruno Buchberger \cite{buchberger} invented the theory of Gr\"{o}bner bases, which is now the cornerstone
of modern computational algebra.  We shall recall briefly the notion of a Gr\"{o}bner basis.
Let $I=\la f_1,f_2, \dots, f_s\ra$ denote the ideal  generated by polynomials $f_1(x_1,\dots,x_n)$, $\ldots$,
$f_m(x_1,\dots,x_n)$, that is,
 the set of all sums
$
\{h_1 f_1+h_2 f_2+\dots + h_s f_s\},
$
where $f_k, h_k$ are polynomials.

A Gr\"{o}bner basis of a given ideal $I$ depends on a term ordering of monomials of $k[x_1,\dots, x_n]$. The two most commonly used term
orders are lexicographic order (lex) and degree reverse lexicographic order
(degrev), defined as follows.
Let $\alphab = (\alpha_1, \dots,\ \alpha_n)$ and
$\betab = (\beta_1, \dots, \beta_n)$ be elements of $\np^n$ ($\np=\N\cup 0$).
 We say that
      $\alphab >_{\rm lex} \betab$  with respect to  lexicographic order if and only if, reading from left to right,
      the first nonzero entry in the $n$-tuple
      $\alphab - \betab \in \Z^n$ is positive; we say that
      $\alphab >_{\rm degrev} \betab$
\   with respect to  degree reverse  lexicographic order
if and only if
      $
      |\alphab| = \sum_{j=1}^n \alpha_j > |\betab| = \sum_{j=1}^n \beta_j
      $
        or
      $
      |\alphab| = |\betab| \
      $
and, reading from right to left, the first nonzero entry in the   $n$-tuple
      $\alphab - \betab \in \Z^n$ is negative.
For $\gamma\in \np$ let  $\xb^{\gamma}$ denote the monomial $x_1^{\gamma_1}x_2^{\gamma_2}\cdots
x_n^{\gamma_n}$.
Fixing a term order  on $\kxn$, any  $f \in \kxn$ may be reordered
in the \emph{standard form} with respect
to the order, that is,
\begin{equation}\label{standard}
f = a_1 \xb^{\alpha_1} + a_2 \xb^{\alpha_2} + \dots + a_s \xb^{\alpha_s},
\end{equation}
where  $\alpha_i \ne \alpha_j$ for $i \ne j$
and $1 \le i,j \le s$, and where, with respect to the specified term order,
$\alpha_1 > \alpha_2 > \cdots > \alpha_s$.
 The \emph{leading term}\index{term!leading} $LT(f)$ of $f$ is the term
      $LT(f) = a_1 \xb^{\alpha_1}$.

Let $f$ and $g$ be from  $\kxn$ with $LT(f) = a \xb^\alphab$ and
$LT(g) = b \xb^\betab$.
The \emph{least common multiple}
 of $\xb^\alphab$ and $\xb^\betab$, denoted
$LCM(\xb^\alphab,\xb^\betab)$, is the monomial
$\xb^\gamma = x_1^{\gamma_1} \cdots x_n^{\gamma_n}$ such that
$\gamma_j = \max(\alpha_j, \beta_j)$, $1 \le j \le n$, and  the
\emph{$S$-polynomial} of $f$ and $g$ is the polynomial
\[
S(f,g)=\frac{\xb^\gamma}{LT(f)}f -\frac{\xb^\gamma}{LT(g)} g.
\]

The following algorithm due to Buchberger \cite{buchberger}  produces a Gr\"{o}bner basis for
the ideal  $I=\la\fs \ra \in
\kxn $.
\begin{itemize}
\item[Step 1.] $G := \{ \fs \}$.
\item[Step 2.] For each pair $g_i, g_j \in G$, $i \ne j$, compute the
        $S$-polynomial $S(g_i, g_j)$ and  compute the remainder $r_{ij}$
of the division  $S(g_i, g_j)$ by $ G$.
\item[Step 3.] If
all $r_{ij}$ are equal to zero, output $G$,
else
add
all nonzero $r_{ij}$ to $G$ and return to Step 2.
\end{itemize}
Nowadays, all major computer algebra systems
(\Mathematica, \Maple, \Reduce, \Singular, \Macaulay\  and  many others) have routines to compute Gr\"{o}bner bases.

A Gr\"{o}bner basis  $G = \{ g_1, \dots, g_m \}$ is called
\emph{reduced}\index{Gr\"obner basis!reduced} if for all $i$, $1 \le i \le m$,
the coefficient of the leading term is 1 and
 no term of $g_i$ is divisible by any  $LT(g_j)$ for $j \ne i$.

It is well-known (see e.g. $\!$\cite{coxlittleshea})  that system (\ref{e:poly.sys})
 has a solution over $\mathbb{C}$ if and only if
the reduced Gr\"obner basis $G$ for $\la \fs \ra$ with respect to any term order on
$\mathbb{C}[x_1, \dots, x_n]$ is different from $\{ 1 \}$.
The Gr\"obner basis theory allows  to find all solutions of system
(\ref{e:poly.sys}) when the system has only finitely many solutions.
In such case  a Gr\"obner basis with respect to
the lexicographic order  is always in a ``triangular" form
(like the Gauss row-echelon form in the case of linear systems)
which means that one has an equation in a single variable, and having solved it one can substitute the roots into an equation in two variables, solve it, etc.

For a field $k$
an  \emph{affine variety} is a subset of $k^n$ that is the solution set of a system of equations  of the form
\eqref{e:poly.sys}, where $f_i$ are polynomials with coefficients
in $k$. It is denoted by $\vv(I)$, where $I$ is the ideal
generated  by $f_1, \ldots, f_m$, $I:=\la f_1,f_2,\dots,f_m\ra$.
A variety is \emph{irreducible} if it is not the union of
finitely many proper subsets, each of which is itself a variety.
Every affine variety $V$  can
be decomposed into finitely many irreducible components, that is  $V$ is expressible as
\begin{equation}\label{gevd}
V = V_1 \cup \dots \cup V_s,
\end{equation}
where each $V_j$ is irreducible and $V_j \not\subset V_k$ if $j \ne k$, and in fact this decomposition is unique up to the ordering of the components $V_j$. Thus to solve \eqref{e:poly.sys} we have
 to find the decomposition \eqref{gevd} for
$V = \vv(I)$.

A radical of the ideal $I$ is the set of polynomials
$\sqrt{I}:=\{f:f^p\in I\text{ for some }p\in\N\}.$

 An ideal $I \subset \kxn$ is called a \emph{primary ideal} if
for any pair $f, g \in \kxn$, $f g \in I$ only if either $f \in I$ or $g^p \in I$ for some $p \in \N$. An ideal $I$ is
primary if and only if $\sqrt{I}$ is prime; $\sqrt{I}$ is called the \emph{associated prime ideal of $I$}. A
\emph{primary decomposition} of an ideal $I \subset \kxn$ is a representation of $I$ as a finite intersection of
primary ideals $Q_j$, $I = \cap_{j = 1}^s Q_j$\,.
The decomposition is called a \emph{minimal} primary decomposition if
the associated prime ideals $\sqrt{Q_j}$ are all distinct and $\cap_{i \ne j} Q_i \not \subset Q_j$ for any $j$.
A minimal primary decomposition of a polynomial ideal always exists, but it is not necessarily unique.

Every ideal $I$ in $\kxn$ has a minimal primary decomposition according to the
Lasker--Noether Decomposition Theorem. 
All such decompositions have the
same number $m$ of primary ideals and the same collection of associated prime ideals.

Minimal associate primes of a polynomial ideal $I=\la f_1,f_2,\dots,f_m\ra$
can be computed using the algorithm proposed by \cite{giannitragerzacharias}, and the varieties of the minimal associate primes give then the irreducible decomposition of the variety $V(I)$ (so give the "solution" to the system $f_1=f_2=\dots f_m=0$).

\subsection{Solving Eq. \eqref{semiout}}\label{subsec:B}
Input is system \eqref{semiin}, and the output is its solution \eqref{semiout}.
\begin{verbatim}
In[20]:= Reduce[{f1 == 0 && f3 == 0 && f5 == 0 && x1 > 0 &&
    x3 > 0 && x5 > 0 && s > 0 && g > 0 && b > 0},
    {x1, x3, x5, s, b, g}] // FullSimplify

Out[20]= {x1 > 0 && x1 == x3 && x3 == x5 && s > 0 && b > 0 &&
    g == s/x1 + b/(1 + x1 + x5)}
\end{verbatim}
\subsection{Minimal associate primes}\label{subsec:C}
Minimal associate primes of ideal \eqref{I} defining the ideals $J_1, J_2, J_3$
of the decomposition \eqref{star} are:
 \begin{verbatim}
 [1]:
   _[1]=x3-x5
   _[2]=x1-x5
   _[3]=2*s*x5+s+b*x5-2*g*x5^2-g*x5
[2]:
   _[1]=x1^3*x3+x1^3*x5+x1^3-2*x1^2*x3*x5+x1^2*x5+x1^2+x1*x3^3
   -2*x1*x3^2*x5+x1*x3^2-2*x1*x3*x5^2-6*x1*x3*x5-x1*x3+x1*x5^3
   -x1*x5+x3^3*x5+x3^3+x3^2+x3*x5^3+x3*x5^2-x3*x5+x5^3+x5^2
   _[2]=b*x3^3+b*x3^2+b*x3*x5+b*x3-b*x5^3-2*b*x5^2-b*x5
   -g*x1^2*x3^2-2*g*x1^2*x3*x5-2*g*x1^2*x3-g*x1^2*x5^2
   -2*g*x1^2*x5-g*x1^2-g*x1*x3^3+g*x1*x3^2*x5-g*x1*x3^2
   +g*x1*x3*x5^2-g*x1*x3-g*x1*x5^3-3*g*x1*x5^2-3*g*x1*x5
   -g*x1+g*x3^3*x5+2*g*x3^2*x5^2+4*g*x3^2*x5+g*x3^2+g*x3*x5^3
   +4*g*x3*x5^2+4*g*x3*x5+g*x3
   _[3]=b*x1*x5+b*x1-b*x3^2-b*x3+g*x1^2*x3+g*x1^2*x5+g*x1^2
   +g*x1*x3^2+2*g*x1*x3-g*x1*x5^2+g*x1-g*x3^2*x5-g*x3*x5^2
   -2*g*x3*x5-g*x5^2-g*x5
   _[4]=b*x1*x3+b*x3-b*x5^2-b*x5-g*x1^2*x3-g*x1^2*x5-g*x1^2
   +g*x1*x3^2-g*x1*x5^2-2*g*x1*x5-g*x1+g*x3^2*x5+g*x3^2
   +g*x3*x5^2+2*g*x3*x5+g*x3
   _[5]=b*x1^2+b*x1-b*x3*x5-b*x5+g*x1^2*x3-g*x1^2*x5+g*x1*x3^2
   +2*g*x1*x3-g*x1*x5^2-2*g*x1*x5+g*x3^2*x5+g*x3^2-g*x3*x5^2
   +g*x3-g*x5^2-g*x5
   _[6]=b^2*x3^2+b^2*x3*x5+b^2*x3+b^2*x5^2+2*b^2*x5+b^2
   +b*g*x3^2*x5+2*b*g*x3^2-b*g*x3*x5^2+b*g*x3*x5+2*b*g*x3
   +b*g*x5^2+2*b*g*x5+b*g+2*g^2*x1*x3^3+2*g^2*x1*x3^2*x5
   +3*g^2*x1*x3^2+2*g^2*x1*x3*x5^2+4*g^2*x1*x3*x5+2*g^2*x1*x3
   +2*g^2*x1*x5^3+5*g^2*x1*x5^2+4*g^2*x1*x5+g^2*x1
   +2*g^2*x3^3*x5+2*g^2*x3^3+4*g^2*x3^2*x5^2+8*g^2*x3^2*x5
   +4*g^2*x3^2+2*g^2*x3*x5^3+8*g^2*x3*x5^2+9*g^2*x3*x5
   +3*g^2*x3+2*g^2*x5^3+5*g^2*x5^2+4*g^2*x5+g^2
   _[7]=2*s*x5+s-b*x1+b*x3+b*x5-2*g*x1*x3-g*x1-g*x3+g*x5
   _[8]=2*s*x3+s+b*x1+b*x3-b*x5-2*g*x1*x5-g*x1+g*x3-g*x5
   _[9]=2*s*x1+s+b*x1-b*x3+b*x5+g*x1-2*g*x3*x5-g*x3-g*x5
   _[10]=s*b+b^2*x1+b^2*x3+b^2*x5+2*b^2+b*g*x1+b*g*x3+b*g*x5
   +2*b*g+2*g^2*x1^2*x3+2*g^2*x1^2*x5+2*g^2*x1^2+2*g^2*x1*x3^2
   +4*g^2*x1*x3*x5+6*g^2*x1*x3+2*g^2*x1*x5^2+6*g^2*x1*x5
   +4*g^2*x1+2*g^2*x3^2*x5+2*g^2*x3^2+2*g^2*x3*x5^2
   +6*g^2*x3*x5+4*g^2*x3+2*g^2*x5^2+4*g^2*x5+2*g^2
   _[11]=s^2+s*g-b^2*x1-b^2*x3-b^2*x5-b^2-b*g*x1-b*g*x3
   -b*g*x5-b*g-2*g^2*x1^2*x3-2*g^2*x1^2*x5-2*g^2*x1^2
   -2*g^2*x1*x3^2-4*g^2*x1*x3*x5-5*g^2*x1*x3-2*g^2*x1*x5^2
   -5*g^2*x1*x5-3*g^2*x1-2*g^2*x3^2*x5-2*g^2*x3^2
   -2*g^2*x3*x5^2-5*g^2*x3*x5-3*g^2*x3-2*g^2*x5^2-3*g^2*x5-g^2
[3]:
   _[1]=g
   _[2]=b
   _[3]=s
\end{verbatim}

\subsection{Checking the conditions of Allwright's theorem in the 3D case}\label{subsec:D}
Here we strongly rely on the paper \cite{allwright}: we use the definitions and notations of that paper.

His equations (5) specialize into our Eq. \eqref{s1}
with the following cast: $n=3$, and for $j=1,2,3: T_j=0, h_j(x)=s+\frac{b}{1+x}, k_j(x)=-gx.$ The quantities and functions defined in this way fulfil conditions (6)--(8) in his paper. As the inverse of $k$ is
$y\mapsto-y/g$ the function $\Phi$ in (9) can be calculated as
\begin{equation*}
\frac{b^2 g+u \left(-b g^2+2 b g s-g^2 s+2 g s^2-s^3\right)-b g^2+3 b g s-b s^2-g^2 s+2 g s^2-s^3}{g u \left(-b g+g^2-2 g s+s^2\right)+g \left(-2 b g+b s+g^2-2 g s+s^2\right)}.
\end{equation*}
The derivative of $\Phi$ is negative for nonnegative arguments $u$ in accordance with the fact that the function is decreasing. Thus we have Case I with the notation of the paper. Further---lengthy---calculations show that the equation $\Phi(\Phi(u))=u$ has one positive (and one negative) real root:
\begin{equation}
u_1=\frac{-g+s+\sqrt{(g+s)^2+4bg}}{2g}>0,\quad u_2=\frac{-g+s-\sqrt{(g+s)^2+4bg}}{2g}<0,
\end{equation}
therefore case (i) of Theorem 1 of the paper applies stating the \emph{global} asymptotic stability of the unique stationary point.

\subsection{Realizations with reversible reactions}\label{subsec:frk}
Consider the equation \eqref{eq:st6D1} for the stationary points of the first 6D model:
\begin{eqnarray}
  s + b x_1 - g x_1 + s x_1 - g x_1^2 + s x_5 - g x_1 x_5 &=& 0 \nonumber\\
  s + s x_1 + b x_3 - g x_3 + s x_3 - g x_1 x_3 - g x_3^2 &=& 0 \label{eq:st6D1}\\
  s + s x_3 + b x_5 - g x_5 + s x_5 - g x_3 x_5 - g x_5^2 &=& 0.\nonumber
\end{eqnarray}
Let us note that the mass action type induced kinetic differential equation
of the reaction in Fig. \ref{fig:sixd1} is has exactly the right hand side equal to the left hand sides of the sbove equations if the reaction rate coefficients have appropriately been chosen. Therefore, based on the results by Orlov and Rozonoer \cite{orlovrozonoer2,tothnagypapp} (or using the recent generalization by Boros \cite{boroswrexistence}) one can conclude that there exists a positive stationary point of the reaction, and thus, of the original (first) 6D model also has one.
\begin{figure}
  \centering
  \includegraphics[width=300pt]{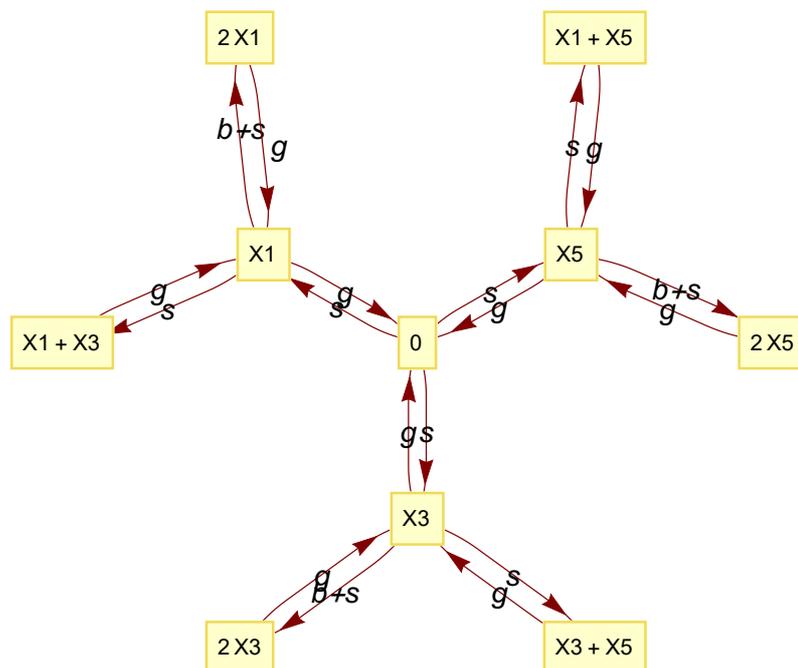}
  \caption{Feinberg--Horn--Jackson graph of a reaction leading to the stationary point of th first 6D model}\label{fig:sixd1}
\end{figure}

The same argument can be applied in the case of the other two models.
\bibliography{Lebar}
\bibliographystyle{spmpsci}
\end{document}